\documentclass[11pt]{amsart}
\usepackage[english]{babel}
\usepackage[T1]{fontenc}
\usepackage[utf8]{inputenc}

\usepackage{graphicx}
\usepackage{amssymb,amscd,amsthm,amsxtra}
\usepackage{latexsym}
\usepackage{epsfig}

\usepackage{fancyhdr}
\usepackage{pgf} 

\usepackage{mathtools}
\usepackage{esint}
\usepackage[active]{srcltx}
\usepackage{color}
\newcommand{\gio}[1]{{\color{black}#1}}
\usepackage{hyperref}
\usepackage{stackengine}
\stackMath
\hypersetup{colorlinks,breaklinks,
            linkcolor=[rgb]{0,0,0},
            citecolor=[rgb]{0,0,0},
            urlcolor=[rgb]{0,0,0}}

\theoremstyle{definition}

\theoremstyle{remark}

\numberwithin{equation}{section}
\newcommand{\be}{\begin{equation}}
\newcommand{\ee}{\end{equation}}

\newcommand{\R}{\mathbb R}
\newcommand{\N}{\mathbb N}

\newcommand{\eps}{\varepsilon}
\newcommand{\loc}{{{\tiny{\mbox{loc}}}}}

\newcommand{\comment}[1]{}
\newcommand\intn{- \hspace{-0.42cm} \int}

\newcommand{\abs}[1]{\big\lvert {#1} \big\rvert}
\newcommand{\norm}[2]{\big\Vert {#1} \big\Vert_{#2}}

\newcommand\ddfrac[2]{\frac{\displaystyle #1}{\displaystyle #2}}
\theoremstyle{definition}
\newtheorem{teo}{Theorem}[section]
\newtheorem{corollary}[teo]{Corollary}
\newtheorem{lemma}[teo]{Lemma}
\newtheorem{theorem}[teo]{Theorem}
\newtheorem{proposition}[teo]{Proposition}
\newtheorem{definition}[teo]{Definition}

\newtheorem{remark}[teo]{Remark}

\usepackage{cfr-lm}
\usepackage[osf]{libertine}

\begin{document}

\title{Liouville theorems and optimal regularity in elliptic equations}

\author[G. Tortone]{Giorgio Tortone}
\address {Giorgio Tortone \newline \indent
	Dipartimento di Matematica, Universit\`a di Pisa \newline \indent
	Largo B. Pontecorvo 5, 56127 Pisa - ITALY}
\email{giorgio.tortone@dm.unipi.it}

\thanks{The author is supported by the European Research Council's (ERC) project n.853404 ERC VaReg - \it Variational approach to the regularity of the free boundaries\rm, financed by the program Horizon 2020. The author is a member of INDAM-GNAMPA}
\subjclass[2020] {35R35, 35B53, 35B65.
}
\keywords{Optimal regularity; elliptic equations; Alt-Caffarelli-Friedman monotonicity formula; anisotropic free-boundary problems; Liouville-type theorems.
}
\begin{abstract}
The objective of this paper is to establish a connection between the problem of optimal regularity among solutions to elliptic PDEs with measurable coefficients and the Liouville property at infinity. Initially, we address the two-dimensional case by proving an Alt-Caffarelli-Friedman type monotonicity formula, enabling the proof of optimal regularity and the Liouville property for multiphase problems.\\
In higher dimensions, we delve into the role of monotonicity formulas in characterizing optimal regularity. By employing a hole-filling technique, we present a distinct ``almost-monotonicity'' formula that implies H\"{o}lder regularity of solutions. Finally, we explore the interplay between the least growth at infinity and the exponent of regularity by combining blow-up and $G$-convergence arguments.
\end{abstract}
\maketitle
\section{Introduction}
The regularity problem for linear elliptic PDEs in its simplest form consists in proving local H\"{o}lder continuity of weak solutions to
\be\label{equation}
\mathrm{div}(A\nabla u)=\frac{\partial}{\partial x_i}\left(a_{ij}(x)\frac{\partial}{\partial x_j} u \right)=0 \quad\mbox{in }B_1,
\ee
where $A =(a_{ij})_{ij}$ is a positive-definite matrix with measurable coefficients satisfying
\be\label{matrix.ell}
 \lambda \abs{\xi}^2 \leq  (A(x)\xi\cdot\xi) \leq L \abs{\xi}^2\quad\text{for $\xi \in \R^n, x \in B_1$},
\ee
for some $0<\lambda\leq L<+\infty$. Initially, such result was obtained in the two-dimensional case by C. B. Morrey \cite{morrey1938} (see also \cite{widman}), but his proof could not be generalized to higher dimension. Finally, in their milestone papers, E. De Giorgi \cite{degiorgi} and J. Nash \cite{nash} solved independently the corresponding problem in every dimension  
(see also \cite{nash} for the same result for parabolic equations) by showing the existence of two positive constants $\alpha^* =\alpha^*(n,\lambda,L), C=C(n,\lambda,L)$ 
such that
$$
|u(x)-u(y)|\leq C|x-y|^{\alpha^*}\norm{u}{L^2(B_1)},\quad\mbox{for every } x,y\in B_{1/2}.
$$
A few years later, J. Moser in \cite{moser1961,moser1960,moser1964} gave a simpler derivation of the H\"{o}lder regularity for weak solutions of both elliptic and parabolic equations by proving the validity of a partial Harnack's type inequality for subsolutions of the problem (see also Proposition \ref{moser.prop}). Despite these results, the dependence of the optimal regularity $\alpha^*=\alpha^*(n,\lambda,L)$ on the elliptic constants is only known in the two-dimensional case. Indeed,
in \cite[Theorem 1]{morrey1938} the author gave a first lower bound for the optimal regularity, that is $\alpha^*(2,\lambda,L) \geq \lambda/(2L)$, but only in \cite[Theorem 1]{piccinini-spagnolo}, L. Piccinini and S. Spagnolo proved that
$$
\alpha^*(2,\lambda,L) = \sqrt{\frac{\lambda}{L}},
$$
which is sharp due to the explicit example of N.G. Meyers \cite[Section 5]{meyers} (see Remark \ref{counter}). The same example gives an upper bound for the optimal regularity of the form
$$
\alpha^*(n,\lambda,L) \leq \sqrt{\left(\frac{n-2}{2}\right)^2 +\frac{\lambda}{L}(n-1)}-\frac{n-2}{2},
$$
for every $n\geq 3$, which is as far as we know the only estimate in higher dimensions.\\

Among other implications of the Harnack's inequality, D. Gilbarg and J. Serrin \cite{gilbargserrin} noticed that it allows to study the asymptotic behavior of solutions to \eqref{equation} near infinity. Indeed, following their iteration argument, in \cite[Section 6]{moser1961} J. Moser proved that, if $u$ is a non-constant solution to
\be\label{equation.global}
\mathrm{div}(A\nabla u)=0 \quad\mbox{in }\R^n,
\ee
with $A=(a_{ij})_{ij}$ satisfying in \eqref{matrix.ell} in the whole $\R^n$, and if its oscillation on ${\mathbb S}^{n-1}_r$ is increasing with respect to $r>0$, then it must grow at infinity as a specific power of $r$ depending only on $n,\lambda$ and $L$. In \cite[Theorem 1]{serrin-weinberger} the authors improved this estimate by showing the validity of the following Liouville-type property\footnote{The formulation of the Liouville-type property can be obtained with some obvious modifications of the original results of \cite[Section 2]{serrin-weinberger}. See also \cite[Theorem 5.2]{lin-ni} for a similar remark.} at infinity: every non-constant solution to \eqref{equation.global} must grow at infinity at least as fast as $|x|^{k^*}$, for some $k^*=k^*(n,\lambda,L)$ (see Definition \ref{def.optimal} for the precise notion of optimal growth).\\
Since the optimal growth $k^*$ is computed by an iteration of the Harnack's inequality, its dependence on the elliptic constant is not explicit and the sharpness of this exponent is also not known. Indeed, by the example of N.G. Meyers is known that
$$
k^*(n,\lambda,L) \leq \sqrt{\left(\frac{n-2}{2}\right)^2 +\frac{\lambda}{L}(n-1)}-\frac{n-2}{2},
$$
for every $n\geq 3$. Although the explicit dependence is not known for general matrix $A$ satisfying \eqref{matrix.ell}, several results have been obtained by prescribing the asymptotic behavior of $A$ at infinity (see \cite[Section 2]{serrin-weinberger}, \cite[Section 5]{lin-ni} and \cite{finn-gilbarg}). Later, in \cite{lin-ni} F. Lin and W. Ni addressed the two-dimensional case showing that
$$
k^*(2,\lambda,L)=\sqrt{\frac{\lambda}{L}},\quad\mbox{ if }\mbox{det}(A)\equiv \mbox{constant}.
$$
This result follows by relating the problem to the Riemann mapping theorem on a manifold with uniformly elliptic metric (see Remark \ref{lin.ni}).\\

Despite the clear relationship between the optimal growth at infinity and the Harnack's inequality, the precise interplay between the Liouville theorem and
the optimal H\"{o}lder regularity of solutions to \eqref{equation} with coefficients satisfying \eqref{matrix.ell} (in particular, the questions whether the two optimal coefficients $k^*(n,\lambda,L)$ and $\alpha^*(n,\lambda,L) $ coincide)  
has not been thoroughly investigated yet, and the present paper aims at taking a first step in this direction.

Moreover, we investigate the validity of Liouville-type theorems and their application to multiphase and free boundary problems, in which we aim to apply our results in a future project (see Remark \ref{rem.future}).

\subsection{Main results and organization of the paper.}
Throughout the paper, we will always denote with $M(\lambda,L)$ the collection of symmetric matrix with bounded, measurable coefficients satisfying the ellipticity condition
$$
\lambda \abs{\xi}^2 \leq  (A(x)\xi\cdot\xi) \leq L \abs{\xi}^2\quad\text{for $\xi \in \R^n, x \in \R^n$}.
$$
We start by addressing the two-dimensional case in which \cite{piccinini-spagnolo,lin-ni} delineated a rather complete picture of the problem. Indeed, we propose an extension of their analysis to the case of multiphase problems, in which the problem is usually stated in terms of a collection of $m$ nonnegative, continuous, subsolutions with disjoint positivity sets.\\

Thus, we start by proving the following Alt-Caffarelli-Friedman type monotonicity formula, which nowadays is the key-tool to the study of local regularity of solutions and free boundaries in multiphase problems.
\begin{theorem}\label{ACFn2}
  Let $B_1\subset \R^2, A \in M(\lambda,L)$ and $(u_i)_{i=1}^m \subset H^1(B_1)$ be a collection of $m$ nonnegative, continuous functions satisfying
  \be\label{subharmonicity}
  -\mathrm{div}(A(x)\nabla u_i)\leq 0 \quad\text{in $B_1$}.
  \ee
  Assume that $u_i$'s have disjoint positivity sets and $u_i(0)=0$ for every $i=1,\dots,m$. Then, the map  
 \be\label{acf.elliptic1}
  \Phi(r) := \prod_{i=1}^m \frac{1}{r^{2k^*_m}}\int_{B_r}{A \nabla u_i\cdot \nabla u_i \,dx},\quad\mbox{with }k^*_m := \frac{m}{2}\sqrt{\frac{\lambda}{L}},
  \ee
  is monotone non-decreasing in $r$, for $r \in (0,1)$.
\end{theorem}
This formula is deeply inspired by the seminal results of \cite{ACF} in which the authors deduce the optimal Lipschitz regularity for solutions to a two-phase free boundary problem. Nowadays, several generalizations of the Alt-Caffarelli-Friedman formula are available in the literature, tailored to deal with equations with regular variable coefficients and right-hand side \cite{CJK, CK,mateACF} and general free boundary problems \cite{CTV, BMPV, Almost-minACF} (see also more recent results \cite{soave2020anisotropic,robinACF} and reference therein).

Nevertheless, the validity of a monotonicity formula for operators with measurable coefficient has not been thoroughly investigated yet, and Theorem \ref{ACFn2} represents the first result in this direction.\\

Inspired by the development in \cite{CTV}, we prove the following Liouville-type result for collection of non-negative, continuous, subharmonic functions with disjoint positivity sets.
\begin{proposition}\label{lio2.introduction}
Let $A \in M(\lambda,L)$ and $(u_i)_{i=1}^m\subset H^1_\loc(\R^2)$ be a collection of $m$ nonnegative, continuous function satisfying
  \be\label{subequation.entire}
  -\mathrm{div}(A(x)\nabla u_i)\leq 0 \quad\text{in $\R^2$},
\ee
  with disjoint positivity sets and such that $u_i(0)=0$ for every $i=1,\dots,m$. Given $k^*_m \in (0,1)$ as in \eqref{acf.elliptic1}, if for some $\gamma \in (0,k^*_m)$ there exists $C>0$ such that
  $$
  u_i(x)\leq C(1+\abs{x}^\gamma)\quad\mbox{in }\R^2, \mbox{ for every }i=1,\dots,m,
  $$
  then there exists an index $j=1,\dots,m$ such that $u_j\equiv 0$.
\end{proposition}
As an immediate consequence, in the case $m=2$, we deduce the growth estimate
$$
k^*(2,\lambda,L)=\sqrt{\frac{\lambda}{L}},\quad\mbox{ for every }A \in M(\lambda,L),
$$
for entire solutions to \eqref{equation}, which completes the analysis carried out in \cite{lin-ni} in the two-dimensional case (see Corollary \ref{lio2.corollary} and Remark \ref{lin.ni}). Moreover, by Remark \ref{counter} we know that Proposition \ref{lio2.introduction} is optimal in the class of matrices $M(\lambda,L)$.
\begin{remark}\label{rem.future}
Regarding the case of free boundary problems involving elliptic operator with measurable coefficient in $M(\lambda,L)$, the only known results require some further regularity assumptions of the results (see \cite{CJK,mateACF}). In the two-dimensional case, Theorem \ref{ACFn2} allows to extend the regularity theory to those variational problems associated to
$$
J(u,B_1)=\int_{B_1}(A\nabla u\cdot \nabla u)\,dx + Q_+ |\{u^+>0\}\cap B_1|
+ Q_- |\{u^-> 0\}\cap B_1| ,
$$
with $A \in M(\lambda,L), Q_+>0, Q_-\geq 0$ and $u^+=\max\{u,0\},u^-=\max\{-u,0\}$. This shall be the subject of a future project, focused on the one-phase and the two-phase problem for general matrix $A\in M(\lambda,L)$.
\end{remark}
One of the primary obstacles to extend those results to the higher-dimensional case is the construction of a monotonicity formula both for \eqref{equation} and for multiphase problems (see Section \ref{section.n} for more details on the known results). Indeed, already in \cite[Section 4]{piccinini-spagnolo} the authors assert that every solution to \eqref{equation} satisfies
$$
\int_{B_r}(A\nabla u\cdot \nabla u)\,dx\leq \left(\frac{r}{R}\right)^{\mu}\int_{B_R}(A\nabla u\cdot \nabla u)\,dx\quad\mbox{for }0<r\leq R<1,
$$
for some positive constant $\mu\leq n-2$. Nevertheless, by Morrey's embedding, this estimate does not imply the H\"{o}lder continuity of solutions of \eqref{equation}, as it happens in the two-dimensional case (see Section \ref{section.n} and Lemma \ref{lem.finn-gilbar} for more details). Inspired by the classical formulation of the Alt-Caffarelli-Friedman formula, we consider a weighted Dirichlet-type integral obtained by considering the fundamental solution $\Gamma$ associated to \eqref{equation}.\\

The first result is a decay estimate of the mentioned energy, which generalizes to higher dimensions $n\geq 3$ the well-known proof of the H\"{o}lder continuity for solutions to \eqref{equation} in $\R^2$ based on the K.O. Widman’s hole-filling technique \cite{widman} and the Caccioppoli's inequality (see Remark \ref{consid1}).
\begin{theorem}\label{new.way.intro}
  Let $A \in M(\lambda,L)$ and $u\in H^1(B_1)$ be a solution of \eqref{equation}. Thus, there exists $\alpha= \alpha(n,\lambda,L) \in (0,1)$ such that
  \be\label{acf.elliptic.partial.intro}
  \int_{B_r}{ \Gamma(x) ( A \nabla u \cdot \nabla u) \,dx}\leq C(n,\lambda,L)\left(\frac{r}{R}\right)^{2\alpha} \int_{B_R} {\Gamma(x) (A\nabla u \cdot \nabla u) \,dx},
  \ee
  for $0<r<R\leq 1/2$. Therefore, $u \in C^{0,\alpha}(B_{1/2})$.
\end{theorem}
\begin{remark}
We emphasize that Theorem \ref{new.way.intro} is independent of Harnack's inequality as established in \cite{degiorgi, nash, moser1961}. Indeed, our result is deeply based on the estimate
\be\label{funda.intro}
C_1(n,\lambda,L)|x|^{2-n}\leq \Gamma(x) \leq C_2(n,\lambda,L)|x|^{2-n} \quad\mbox{in }\R^n,
\ee
which was directly proved by E.B. Fabes and D.W. Stroock in \cite{fabes-stroock}, without using either the validity of an Harnack inequality or the H\"{o}lder regularity of solutions to \eqref{equation} (see Section \ref{section.prel} for a detailed discussion on the role of the fundamental solution).
\end{remark}
It is clear that the H\"{o}lder exponent in Theorem \ref{new.way.intro} is far from being optimal in the class $M(\lambda,L)$. Thereby, we started a finer analysis of the optimal regularity in $M(\lambda,L)$ by showing the existence of a formal link between the optimal regularity and the optimal growth at infinity.
\begin{theorem}\label{liovu.intro}
  Let $A \in M(\lambda,L)$ and $u \in H^1(B_1)$ be a weak-solution to
$$
  -\mathrm{div}(A(x)\nabla u)= 0 \quad\text{in $B_1$},
$$
satisfying $\lVert u\rVert_{L^\infty(B_1)}\leq 1$.
Then $u\in C^{0,\alpha}(\overline{B_{1/2}})$, for every $\alpha \in (0,k^*(n,\lambda,L))$.
\end{theorem}
This result is obtained by a contradiction argument in which we construct a blow-up sequence of solutions to \eqref{equation} associated to some $\overline{A}_j \in  M(\lambda,L)$ (see \eqref{blowup} for the precise definitions). Unfortunately, the presence of matrices in $M(\lambda,L)$ does not allow to prove the compactness of weak solutions in the $H^1$-strong topology and explicitly construct the equation satisfied by the blow-up limit.\\
Hence, we overcome this obstacle by applying the notion of $G$-convergence introduced in \cite{spagnolo1,spagnolo2,marino-spagnolo,spagnolo-survey} to our blow-up analysis, proving $G$-compactness of the problems and the existence of a blow-up limit satisfying \eqref{equation} for some matrix in $M(\lambda,L)$.
\begin{remark}\label{consid2}
Theorem \ref{new.way.intro} and Theorem \ref{liovu.intro} suggest that in higher dimensions the optimal regularity and the least growth at infinity could be deduced by the existence of $\alpha^* \in (0,1)$ for which the map
$$
  \Phi(r) := \frac{1}{r^{2\alpha^*}}\int_{B_r}{\Gamma(x)(A \nabla u\cdot \nabla u) \,dx}
$$
is monotone non-decreasing for $r\in (0,1)$. Indeed, by combining the monotonicity result with the Poincar\'{e} and Caccioppoli's inequalities, one could immediately derive that
$$
  \sup_{x_0 \in B_{1/2}}\sup_{r \leq 1/2} \frac{1}{r^{n+2\alpha^*}}  \int_{B_r(x_0)}{(u-u_{B_r(x_0)})^2 \,dx} \leq C_n \left(\frac{C_2}{C_1}\right)^{\frac{n-2}{2}} \norm{u}{L^2(B_1)}^2,
$$
with $C_1,C_2>0$ defined by \eqref{funda.intro}. Moreover, in view of Theorem \ref{spagno1} and the uniqueness of the solution of the Dirichlet problem associated to \eqref{equation}, it is sufficient to prove the monotonicity result for a family of matrices $G$-dense in $M(\lambda,L)$ (see Section \ref{section.prel} for more details on the $G$-convergence).
\end{remark}
Our analysis implies that determining the asymptotic behavior of the fundamental solutions near the singularity (resp. at infinity) is the key point to deducing the optimal regularity (resp. the optimal growth at infinity) in $M(\lambda,L)$. Indeed, even though in \cite{kozlov} S.M. Kozlov developed an asymptotic analysis of the fundamental solution near the origin by means of the concept of $G$-convergence, his analysis requires a priori the uniqueness of the blow-up limit of the operator $A \in M(\lambda,L)$, which is not always true in our setting (see also \cite{avella1,avella2} for similar results in the homogenization theory).\\

Therefore, we conclude the study by showing an application of our analysis under strong assumptions on the matrix $A\in M(\lambda,L)$ (see \eqref{G-limit}). Indeed, by using the concept of blow-down sequences, we show in Proposition \ref{liouv.Finn} that the optimal Liouville exponent $k^*_A$ of a fixed matrix (see Definition \ref{def.optimal}) can be obtained in terms of its blow-down limit at infinity (see also \cite[Theorem 2]{moser-struwe} for more results in this directions). In view of that, we improve a known estimate by \cite[Section 2]{serrin-weinberger} by showing that if a matrix $A$ converges to some constant matrix at infinity, then $k^*_A=1$.\\

The paper is organized as follows: in Section \ref{section.prel} we recall some preliminary results on the notions of $G$-convergence and fundamental solution by stressing the connection with the validity of the Harnack's inequality.
In Section \ref{section.2} we address the two-dimensional case by proving Theorem \ref{ACFn2} and Proposition \ref{lio2.introduction}, from which Corollary \ref{lio2.corollary} follows by simple arguments. Finally, in Section \ref{section.n} we discuss the higher-dimensional case by discussing the interplay between optimal regularity and optimal growth at infinity in the proof of Theorem \ref{new.way.intro}. In the end, in Proposition \ref{liouv.Finn} we provide an application of our asymptotic analysis to the problem of the optimal growth at infinity for a fixed matrix in $M(\lambda,L)$.

\section{Preliminary results}\label{section.prel}
In this Section we start by introducing the concept of the  $G$-convergence and recalling some well-known result for the fundamental solution of uniformly elliptic divergence operator with bounded coefficients. In the end, we show some technical results that will be used through the paper.
\subsection{\bf{$\mathbf{G}$-convergence.}} The theory of the $G$-convergence has been the subject of a systematic study by S. Spagnolo et al. in the series of papers \cite{spagnolo1,spagnolo2,marino-spagnolo} (see also the survey \cite{spagnolo-survey} and reference therein for an extensive discussion) in which they introduce an abstract scheme for the convergence in energy of elliptic and parabolic operators. We recall now its definition and some results.\\
The following definition (see \cite[Theorem 4]{spagnolo1}) refers to the $G$-convergence of elliptic operators.
\begin{definition}
 Let $A_k \in M(\lambda,L)$, we say that $A_k$ is $G$-convergent to some $A \in M(\lambda,L)$ in $\Omega\subset \R^n$ if and only if, for some $t>0$, we have
 $$
 (\mathrm{div}(A_k\nabla)+t)^{-1}f \to
  (\mathrm{div}(A\nabla)+t)^{-1}f \quad\mbox{strongly in }L^2(\Omega),
 $$
 for every $f \in H^{-1}(\Omega)$. Alternatively, if $\Omega$ is bounded, the $G$-convergence is equivalent to
 $$
 (\mathrm{div}(A_k\nabla))^{-1}f \to
  (\mathrm{div}(A\nabla))^{-1}f \quad\mbox{strongly in }L^2(\Omega),
 $$
 for every $f \in H^{-1}(\Omega)$.
\end{definition}
For the sake of simplicity, through the paper we will also use the notation $\xrightarrow{\text{ $G$ }}$ to denote the $G$-convergence of a sequence of matrices in $M(\lambda,L)$. As pointed out in \cite[Section 5]{spagnolo1}, the $G$-convergence is highly suited to the study of local solutions of homogeneous divergence operator associated to some $A\in M(\lambda,L)$.
\begin{theorem}\label{spagno1}
  Let $A_k\in M(\lambda,L), u_k \in  H^1_\loc(\Omega), u\in H^1_\loc(\Omega)$ and $\tilde{\Omega}\subset \Omega$ be open. Suppose that:
  \begin{itemize}
    \item $A_k$ is $G$-convergent to $A$ in $\tilde{\Omega}$;
    \item for every $k>0$, the function $u_k$ is a weak solution of
        $$
        -\mathrm{div}(A_k\nabla u_k)=0 \quad\mbox{in }\tilde{\Omega};
        $$
    \item for every compact set $K\subset\subset \tilde{\Omega}$ we have
    $$
    \lim_{k\to \infty}\int_{K}(u_k -u)^2\,dx = 0.
    $$
  \end{itemize}
  Thus, $\mathrm{div}(A\nabla u)=0$ in $\tilde{\Omega}$ in the weak sense.
\end{theorem}
Even though the convergence of matrices $(A_k)_k$ in $L^1_\loc$ implies the $G$-convergence, the latter is not equivalent to any convergence of the corresponding coefficients (see the approximation result in \cite{marino-spagnolo}). Finally, since by \cite[Proposition 3]{spagnolo1} every sequence in $M(\lambda,L)$ admits a $G$-convergent subsequence, Theorem \ref{spagno1} implies the following result.
\begin{proposition}{\cite[Corollary 1]{spagnolo1}}\label{closed}
  The set of all local solution
  $$
  S(\lambda,L)=\left\{u \in H^1_\loc(\Omega)\colon \mathrm{div}(A\nabla u)=0 \mbox{ in }\Omega,\mbox{ with }A \in M(\lambda,L)\right\}
  $$
  is closed in $L^2_\loc(\Omega)$.
\end{proposition}
\subsection{\bf{Fundamental solution.}} In the pioneering work \cite{LSW} W. Littman, G. Stampacchia and H.F. Weinberger 
proved that for $n\geq 3$, every solution to \eqref{equation} approaching plus infinity at the origin is comparable with the fundamental solution of the Laplacian $|x|^{2-n}$, and that a fundamental solution of a divergence form operator associated to $A\in M(\lambda,L)$ satisfying
\be\label{estimate.fundamental}
C_1(n,\lambda,L)|x|^{2-n}\leq \Gamma(x) \leq C_2(n,\lambda,L)|x|^{2-n} \quad\mbox{in }\R^n
\ee
exists (a similar result for parabolic equation was first proved by D.G. Aronson in \cite{aronson}).\\
Starting from this work, in \cite{gruter-widman,kozlov,serrin-weinberger} the authors refined the asymptotic analysis of the fundamental solutions of divergence elliptic equations and they addressed a similar analysis for the Green function on bounded domains. In contrast with the other works, in \cite{kozlov} S.M. Kozlov considered the asymptotic behavior of the fundamental solution near the origin by means of the concept of $G$-convergence (the asymptotic at infinity is obtained only for operators with smooth, quasiperiodic coefficients).\\
We notice that all the results in \cite{gruter-widman,kozlov,serrin-weinberger} are based on the validity of the estimate \eqref{estimate.fundamental} of \cite{LSW}, which relies on the validity of the Harnack's inequality of solutions of \eqref{equation} (see the classical works \cite{degiorgi, nash, moser1960}).\\

Nevertheless, years later in \cite{fabes-stroock} E.B. Fabes and D.W. Stroock, by modifying and pursuing the ideas of J. Nash in \cite{nash}, directly established \eqref{estimate.fundamental} without using either the validity of an Harnack inequality or the H\"{o}lder regularity for solutions of \eqref{equation}. 
\begin{proposition}{\cite{fabes-stroock}}\label{estimate.fund.prop}
Let $A \in M(\lambda,L)$ and $\Gamma$ be the fundamental solution associated to $A$ in $\R^n$. Then
$$
C_1|x|^{2-n}\leq \Gamma(x) \leq C_2|x|^{2-n} \quad\mbox{in }\R^n,
$$
for some constant $C_1,C_2>0$ depending only on $n,\lambda$ and $L$.
\end{proposition}
The following is a straightforward combination of the Caccioppoli's inequality and Proposition \ref{estimate.fund.prop} that will be useful in Section \ref{section.n}.
\begin{lemma}\label{fundamental.lemma}
Let $A\in M(\lambda,L)$ and $\Gamma$ be the fundamental solution associated to $A$. Then, for every $r>0$ we have
$$
\int_{B_{{\frac32}r}\setminus B_r}|\nabla \Gamma|^2\,dx \leq C r^{2-n},
$$
for some constant $C$ depending only on $n,\lambda $ and $L$.
\end{lemma}
\begin{proof}
Since the proof coincides with the one of the classical Caccioppoli's inequality, we just sketch the classic strategy.
Therefore, let $\eta \in C^\infty_c(B_{2r}\setminus B_{r/2})$  be such that $0\leq \eta \leq 1, \eta \equiv 1$ on $B_{\frac32 r}\setminus B_r$ and
  $$
  |\nabla \eta|\leq \frac{C(n)}{r}.
  $$
  By testing the equation satisfied by $\Gamma$ with $\Gamma\eta^2$ in $B_{2r}\setminus B_{r/2}$ we get
  $$
  \int_{B_{2r}\setminus B_{r/2}}\eta^2|\nabla \Gamma|^2\,dx\leq \frac{C(n,\lambda,L)}{r^2}  \int_{B_{2r}\setminus B_{r/2}}\Gamma^2\,dx \leq C(n,\lambda,L) r^{2-n}
  $$
  where in the last inequality we used Proposition \ref{estimate.fund.prop}.
\end{proof}
\subsection{\bf{Uniformly elliptic divergence operator. }}We conclude the Section by recalling some results for solutions of divergence form equation associated to $A\in M(\lambda,L)$ that will be used through the paper.\\

The following Proposition is the well-known Moser's iteration argument introduced in \cite[Theorem 2]{moser1961}, that allows to prove $L^p$ estimate for a subsolution of \eqref{equation} in terms of its $L^2$-norm (see the proof of \cite[Proposition 8.20]{GiaMart}).
\begin{proposition}\label{moser.prop}
  Let $A\in M(\lambda,L)$ and $u \in H^1(B_1)$ be a subsolution of \eqref{equation} and $r\in (0,1/2)$. Then, there exists $C=C(n,\lambda,L)$ such that
\be\label{moser.eq}
\left(\frac{1}{r^n}\int_{B_{\frac32 r}\setminus B_{r}} u^{k_i}\,dx\right)^{\frac{2}{k_i}} \leq \frac{C}{r^n}\int_{B_{2r}\setminus B_r}u^2 \,dx, \quad\mbox{with }k_i=\frac{2n^i}{(n-2)^i},
\ee
for every $i\in 1+\N$.
\end{proposition}
Finally, we conclude the Section with the following Lemma in which we decompose a quadratic form associated to $A\in M(\lambda,L)$ in terms of its ``radial'' and ``spherical'' part.
\begin{lemma}\label{tangential}
Let $A\in M(\lambda,L)$ and $\nu(x)=x/|x| \in  {\mathbb S}^{n-1}$. Then for every vector field $\xi \in \R^n$ we have
\be\label{polar}
(A\xi\cdot \xi) \geq \frac{1}{L}(A\xi\cdot \nu)^2 +\lambda \abs{\xi_{\mathbb{S}}}^2,\ee
where $\xi_{\mathbb{S}}$ is the tangential component of $\xi$ on ${\mathbb S}^{n-1}$.
\end{lemma}
\begin{proof}
Set
$$
\xi^A_{\mathbb{S}} :=\xi -\frac{(A\xi\cdot \nu)}{(A\nu\cdot \nu)}\nu,\quad \xi_{\mathbb{S}} := \xi-(\xi\cdot \nu)\nu
$$
where $\xi_{\mathbb{S}} \perp \nu$ and $A\xi_{\mathbb{S}}^A \perp \nu$. Then
$$
  (A\xi\cdot \xi)= (A\xi_{\mathbb{S}}^A\cdot \gio{\xi_{\mathbb{S}}^A}) + \frac{(A\xi\cdot \nu)^2}{(A\nu\cdot \nu)} \geq \lambda |\xi_{\mathbb{S}}^A|^2 + \frac{1}{L}(A\xi\cdot \nu)^2.
$$
On the other hand
$$
\xi^A_{\mathbb{S}} = \xi_{\mathbb{S}} +\left[(\xi\cdot \nu) -\frac{(A\xi\cdot \nu)}{(A\nu\cdot \nu)}\right]\nu
$$
and since $\xi_{\mathbb{S}} \perp \nu$, we get $|\xi^A_{\mathbb{S}}|^2\geq |\xi_S|^2$, as we claimed.
\end{proof}
\section{The $2$-dimensional case}\label{section.2}
In this Section we address the problem in $\R^2$ by proving Theorem \ref{ACFn2} and Proposition \ref{lio2.introduction}, from which Corollary \ref{lio2.corollary} follows by simple arguments.\\

We start by showing the monotonicity of the Alt-Caffarelli-Friedman type formula introduced in \eqref{acf.elliptic1}.
\begin{proof}[Proof of Theorem \ref{ACFn2}]
Since $u_i \in H^1(B_1)$, all the integrals in \eqref{acf.elliptic1} are locally absolutely
continuous functions in $r$, for $r \in (0,1)$. Thus, the result follows once we prove that $(\log\Phi(r))' \geq 0$ for almost every
$r \in (0,1)$. Hence, by direct computations we have
$$
\frac{\,d}{\,dr}\log\Phi(r) = \sum_{i=1}^m\ddfrac{\int_{\partial B_r}{ A \nabla u_i \cdot \nabla u_i \,d\mathcal{H}^1}}{\int_{B_r}{A \nabla u_i\cdot \nabla u_i\,dx}}
-2m\frac{\alpha^*_m}{r},\quad\mbox{with }\alpha^*_m := \frac{m}{2}\sqrt{\frac{\lambda}{L}}.
$$
Since \gio{$u_i$ is non-negative, it holds $2\langle A \nabla u_i, \nabla u_i\rangle \leq \mathrm{div}(A\nabla u_i^2)$ in a weak sense}, and by H\"{o}lder inequality we get
\be\label{denominatore}
\int_{ B_r}{A \nabla u_i\cdot \nabla u_i\,dx} \leq
\int_{\partial B_r}{u_i (A \nabla u_i\cdot \nu) \,d\mathcal{H}^1}
\leq
\left(\int_{\partial B_r}{u_i^2 \,d\mathcal{H}^1}\right)^{1/2} \left(\int_{\partial B_r}{(A \nabla u_i \cdot \nu)^2  \,d\mathcal{H}^1}\right)^{1/2}.
\ee
On the other side, by applying Lemma \ref{tangential} with $\xi = \nabla u_i$, we get
\begin{align}\label{numeratore}
\begin{aligned}
\int_{\partial B_r}{A \nabla u_i \cdot \nabla u_i \,dx} &\geq
\frac{1}{L}\int_{\partial B_r}{(A \nabla u_i \cdot \nu)^2  \,d\mathcal{H}^1} +\lambda
\int_{\partial B_r}{\abs{\nabla_{\mathbb{S}} u_i}^2 \,d\mathcal{H}^1}\\
&\geq 2\sqrt{\frac{\lambda}{L}} \left(\int_{\partial B_r}{(A \nabla u_i \cdot \nu)^2  \,d\mathcal{H}^1}\right)^{1/2}\left(\int_{\partial B_r}{{\abs{\nabla_{S} u_i}}^2 \,d\mathcal{H}^1}\right)^{1/2},
\end{aligned}
\end{align}
where $\nabla_{\mathbb{S}}$ represents the tangential gradient on $\mathbb{S}^1$. Now, let $u_i^{(r)}\colon {\mathbb{S}}^{1}\to \R$ be such that $u_i^{(r)}(x)=u_i(rx)$ and consider $\Omega_i^r=\{u_i^{(r)}>0\}\cap {\mathbb{S}}^1$. Thus, by \eqref{denominatore} and \eqref{numeratore} we get
\begin{align}\label{ray}
\begin{aligned}
\frac{\,d}{\,dr}\log\Phi(r) &\geq 2\sqrt{\frac{\lambda}{L}}  \frac1r \sum_{i=1}^m\left(\ddfrac{\int_{\mathbb{S}^{1}}{{\abs{\nabla_{\mathbb{S}} u_i^{(r)}}}^2 \,d\mathcal{H}^1}}{\int_{\mathbb{S}^{1}}{(u_i^{(r)})^2 \,d\mathcal{H}^1}}\right)^{1/2}\!
-2m\frac{\alpha^*_m}{r}\\
&\geq \frac{2m}{r}\sqrt{\frac{\lambda}{L}}  \left(\sum_{i=1}^m\frac{\Lambda_1(\Omega_i^r)^{1/2}}{m}
-\frac{m}{2}\right),
\end{aligned}
\end{align}
where for every $\mathcal{H}^1$-measurable set $\omega\subset \mathbb{S}^1$ we denote with $\Lambda_1(\omega)$ the first Dirichlet eigenvalue associated to subsets of $\mathbb{S}^1$, that is
$$
\Lambda_1(\omega):=\min\left\{\ddfrac{\int_{\mathbb{S}^{1}}{{\abs{\nabla_{\mathbb{S}} v}}^2 \,d\mathcal{H}^1}}{\int_{\mathbb{S}^{1}}{v^2 \,d\mathcal{H}^1}}\colon v \in H^1(S^1),\quad\mathcal{H}^1(\{v\neq 0\}\setminus \omega) =0\right\}.
$$
Therefore, by defining with $\mathcal{P}^m$ the set of $m$-partitions $\omega=(\omega_1,\dots,\omega_m)$ of $\mathbb{S}^{1}$, we finally get
\be
\frac{\,d}{\,dr}\log\Phi(r) \geq \frac{2m}{r}\sqrt{\frac{\lambda}{L}}  \left[\inf_{\omega \in \mathcal{P}^m} \sum_{i=1}^m \frac{\Lambda(\omega_i)^{1/2}}{m}
-\frac{m}{2}\right],
\ee
which corresponds to the minimization problem associated to the same result for the classical Laplacian (see \cite{ACF,CTV}). Thus, by \gio{following the direct computations in \cite[Lemma 4.2]{CTV}}, it follows that
$$
\inf_{\omega \in \mathcal{P}^m} \sum_{i=1}^m \frac{\Lambda(\omega_i)^{1/2}}{m}
=\frac{m}{2},
$$
which immediately implies the claimed result.
\end{proof}
Now, by combining the Caccioppoli inequality with Theorem \ref{ACFn2}, we can prove the Liouville-type results Proposition \ref{lio2.introduction}, for vector of sub-harmonic functions.
\begin{proof}[Proof of Proposition \ref{lio2.introduction}]
For the sake of contradiction, assume that $u_i \not \equiv 0$ for every $i=1,\dots,m$, \gio{thus there exists a radius $r_0>0$ such that $\Phi(r_0)>0$}. On the other hand, by the Caccioppoli inequality
  $$
\int_{B_r}{\langle A \nabla u_i, \nabla u_i\rangle  \,dx} \leq \frac{C}{r^2}\int_{B_{2r}\setminus B_r}u_i^2\,dx,\qquad \mbox{for every }r>0.
  $$
  and for some $C=C(\lambda,L)$. On the other hand, fixed $0<r_0\leq r$ we infer that
  $$
  0\leq \Phi(r_0)\leq \Phi(r)\leq \frac{1}{r^{2m\alpha^*}}\prod_{i=1}^m \frac{C}{r^2}\left(\int_{B_{2r}\setminus B_r}u_i^2\,dx\right) \leq C r^{2m(\gamma-\alpha_*)},
  $$
  which leads to a contradiction for $r$ large enough.
\end{proof}
Naturally, Proposition \ref{lio2.introduction} implies an explicit estimate for the least growth at infinity $k^*(2,\lambda,L)$ of entire solutions of \eqref{equation} with $n=2$.
\begin{corollary}\label{lio2.corollary}
  Let $A \in M(\lambda,L)$ and $u \in H^1_\loc(\R^2)$ be a solution to
  \be\label{equation.entire}
  -\mathrm{div}(A(x)\nabla u) = 0 \quad\text{in $\R^2$}.
  \ee
  Given $k^*(2,\lambda,L)= \sqrt{\frac{\lambda}{L}}$, if for some $\gamma \in (0,k^*)$ there exists $C>0$ such that
  \be\label{e:gro}
  |u|(x)\leq C(1+\abs{x}^\gamma)\quad\mbox{in }\R^2,
  \ee
  then $u$ is constant.
\end{corollary}
\begin{proof}
\gio{Let $\beta \in \R$ and set $u_1:=\max\{u-\beta,0\},u_2:=\max\{\beta-u,0\}$ being with disjoint supports. Naturally, in light of \eqref{e:gro}, there exists $\beta$ such that neither $u_1\equiv 0$ and $u_2\equiv 0$. On the other hand, they fulfill the assumptions of Proposition \ref{lio2.introduction} in the case $m=2$, which implies that either $u_1\equiv 0$ or $u_2\equiv 0$, in contradiction with the choice of $\beta$.}
\end{proof}
\begin{remark}\label{lin.ni}
  In \cite{lin-ni}, the authors already proved the existence of an optimal least growth for non-trivial entire solutions of \eqref{equation.entire} in the case $\mbox{det}(A)$ is constant, by relating the problem to the Riemann mapping theorem for the manifold $(\R^2 ,g)$ with $g$ a uniformly elliptic metric. In this direction, Proposition \ref{lio2.introduction} answers the open question on the ``optimal gap phenomenon'' of Question 1 of \cite{lin-ni} by proving the sharpness of the least growth of order $\sqrt{\lambda/L}$. Indeed, this result finalizes the analysis carried out in \cite{lin-ni} in $\R^2$ and it extend the existence of an optimal least growth at infinity to the case of non-negative subharmonic functions with disjoint positivity sets.
\end{remark}
\begin{remark}\label{counter}
The optimality of Proposition \ref{lio2.introduction} and Corollary \ref{lio2.corollary} follows by rephrasing the example of \cite{meyers}. More precisely, given $0<\lambda\leq L$ set $A=(a_{ij})_{ij}$ with
$$
a_{ij}(x):= \lambda \delta_{ij} + \left(L-\lambda\right)\frac{x_ix_j}{|x|^2}.
$$
By passing to the polar coordinates $(r,\theta)$ in $\R^2$, it follows that a solution $u$ of \eqref{equation.entire} with $A$ as before satisfies
$$
\gio{\frac{L}{r} \partial_r(r\partial_ru) + \frac{\lambda}{r^2}\partial^2_\theta u = 0.}
$$
Therefore, given
$$
u(x)=|x|^{\sqrt{\frac{\lambda}{L}}}\frac{x_2}{|x|} = r^{\sqrt{\frac{\lambda}{L}}} \sin(\theta),
$$
we have that $u$ is solution of \eqref{equation.entire} satisfying the optimal least growth at infinity. On the other hand, we can generalize the previous example to the case of junction points: indeed, for every $m \in \N, m\geq 2$ the collection $(u_i)_{i=1}^m$ defined as
$$
u_i(x)=
\begin{cases}
r^{\frac{m}{2}\sqrt{\frac{\lambda}{L}}}\left|\sin\left(\frac{m}{2}\theta\right)\right| &\mbox{for }\theta \in \left[\frac{2\pi}{m} (i-1),\frac{2\pi}{m} i\right]\\
0 & \mbox{otherwise},
\end{cases}
,\quad\mbox{for }i=1,\dots,m
$$
satisfies the optimal least growth at infinity.
\end{remark}
\section{The $n$-dimensional case}\label{section.n}
In this Section we address the higher-dimensional case $n\geq 3$.
As we mentioned in the Introduction, in \cite[Section 4]{piccinini-spagnolo} the authors claimed the existence of a decay estimate for the Dirichlet energy associated to solutions to \eqref{equation}, which does not imply the H\"{o}lder continuity of solutions.\\ More precisely, by following the same computations of Section \ref{section.2}, it can be proved the following decay estimate of the Dirichlet energy in the spirit of \cite{piccinini-spagnolo,finn-gilbarg}.
\begin{lemma}\label{lem.finn-gilbar}
 Let $A \in M(\lambda,L)$ and $u_1, u_2 \in H^1(B_1)$ be two nonnegative functions such that
  \be\label{subharmonicity}
  -\mathrm{div}(A(x)\nabla u_i)\leq 0 \quad\text{in $B_1$}.
  \ee
  Assume that $u_1\cdot u_2 \equiv 0$ and $u_1(0)=u_2(0)=0$. Then, the map
  $$
r \mapsto \frac{1}{r^{n-2+2\mu}}\int_{B_r}(A\nabla u\cdot \nabla u)\,dx,\quad\mbox{with }\mu:=\sqrt{\frac{\lambda}{L}(n-1)}-\frac{n-2}{2}
$$
is monotone non-decreasing in $r$, for $r \in (0,1)$.
\end{lemma}
\begin{proof}
We simply sketch few steps of the proof. By following the lines of Theorem \ref{ACFn2}, if we set
$$
\Phi(r) := \prod_{i=1}^2\frac{1}{r^{n-2+2\mu}}\int_{B_r}(A\nabla u_i\cdot \nabla u_i)\,dx,
$$
we get
$$
\frac{\,d}{\,dr}\log\Phi(r) \geq
 2\sqrt{\frac{\lambda}{L}}\frac1r \sum_{i=1}^2 \left(\ddfrac{\int_{{\mathbb S}^{n-1}}{{\abs{\nabla_{\mathbb{S}} u_i^{(r)}}}^2 \,d\mathcal{H}^{n-1}}}{\int_{{\mathbb S}^{n-1}}{(u_i^{(r)})^2 \,d\mathcal{H}^{n-1}}}\right)^{1/2}\!
-\frac{4}{r}\sqrt{\frac{\lambda}{L}(n-1)}
$$
where $u_i^{(r)}\colon {\mathbb S}^{n-1}\to \R$ is such that $u_i^{(r)}(x)=u_i(rx)$. Thus, for every $\mathcal{H}^{n-1}$-measurable set $\omega\subset {\mathbb S}^{n-1}$ we denote with $\Lambda_1(\omega)$ the first Dirichlet eigenvalue of $\omega$ associated to the Laplace-Beltrami operator of ${\mathbb S}^{n-1}$. Then, by defining with $\mathcal{P}^2$ the set of $2$-partitions $(\omega_1,\omega_2)$ of ${\mathbb S}^{n-1}$, we finally get

$$
 \sum_{i=1}^2 \left(\ddfrac{\int_{{\mathbb S}^{n-1}}{{\abs{\nabla_{\mathbb{S}} u_i^{(r)}}}^2 \,d\mathcal{H}^{n-1}}}{\int_{{\mathbb S}^{n-1}}{(u_i^{(r)})^2 \,d\mathcal{H}^{n-1}}}\right)^{1/2} \geq  2\inf_{\omega \in \mathcal{P}^2} \frac{(\Lambda_1(\omega_1))^{1/2}+(\Lambda_1(\omega_2))^{1/2}}{2}\geq 2\sqrt{n-1}.
$$
\gio{Indeed, by exploiting the validity of Faber-Krahn type inequality on ${\mathbb S}^{n-1}$, we can restrict the minimization problem to the class of spherical caps in ${\mathbb S}^{n-1}$. Ultimately, by  adapting the arguments in \cite{ACF,friedland} (see also the detailed computation in \cite[Section 2.2.3]{LSnoris}) for the convexity of the first eigenvalue among spherical caps, it is possible to prove optimal balancing and show that the minimum is achieved by splitting the spheres in two disjoint hemispheres.}
\end{proof}
It is natural to expect that the previous estimate is far from being optimal. Indeed, it is easy to see that the exponent $\mu$ leads to H\"older continuity by Morrey's embedding for specific choices of $\lambda$ and $L$, depending on $n$ only, in the range $n\in [2,6]$. Therefore, we start by proving Theorem \ref{new.way.intro}, in which the H\"{o}lder continuity of solutions to \eqref{equation} is deduced through the study of a weighted Dirichlet-type energy similar to the one of the Alt-Caffarelli-Friedman formula of \cite{ACF}.
\begin{proof}[Proof of Theorem \ref{new.way.intro}]
Let $\eta_\eps \in C^\infty_c(B_\eps)$ be a standard mollifier and consider the regularized matrix $A^\eps := (a_{ij}^\eps)_{ij} \in M(\lambda,L)$ with entries $a_{ij}^\eps :=a_{ij}\ast\eta_\eps$. Then, set $u_\eps, v_\eps\in C^\infty(B_1)$ be such that
$$
\begin{cases}
\mathrm{div}(A^\eps \nabla u_\eps)= 0 &\mbox{in }B_1\\
u_\eps = u &\mbox{on }\partial B_1
\end{cases},\qquad v_\eps := u_\eps - \intn_{B_r}u_\eps\,dx,
$$
and $\Gamma_\eps$ to be the fundamental solution associated to $A^\eps$ (see \cite{LSW} for a detailed discussion on the approximation of $\Gamma$). By adapting the lines of \cite[Remark 1.5]{CJK}, we deduce that
$$
\int_{B_r}\Gamma_\eps(A^\eps\nabla v_\eps\cdot \nabla v_\eps)\,dx \leq \frac{C(n,\lambda,L)}{r^n}\int_{B_{2r}}v_\eps^2\,dx\quad\mbox{for every }r \in (0,1/2),
$$
which implies, as $\eps\to 0^+$, that
\be\label{e:cacco}
\int_{B_r}\Gamma(A\nabla u\cdot \nabla u)\,dx \leq \frac{C(n,\lambda,L)}{r^n}\int_{B_{2r}}(u-u_{B_r})^2\,dx\quad\mbox{for every }r \in (0,1/2).
\ee
Now, fix $r \in (0,1/2)$ and $\phi \in C^\infty_c(B_{\frac32 r})$ be such that $0\leq \phi\leq 1$ and $\phi\equiv 1$ in $B_r$. Then, integrating by parts we get
\begin{align*}
2 \int_{B_{\frac32 r}} \Gamma_\eps(A^\eps\nabla v_\eps \cdot \nabla v_\eps)\phi^2 \,dx&\leq -\int_{B_{\frac32 r}}(A^\eps\nabla (v_\eps^2) \cdot \nabla (\Gamma\phi^2))\,dx\\
&\leq  - \int_{B_{\frac32 r}} \phi^2 (A^\eps\nabla (v_\eps^2)  \cdot \nabla \Gamma_\eps)\,dx - 2 \int_{B_{\frac32 r}\setminus B_r} \phi \Gamma_\eps (A^\eps\nabla (v_\eps^2) \cdot \nabla \phi)\,dx,
\end{align*}
\gio{where we used that $\mathrm{div}(A^\eps \nabla(v_\eps^2))=2 (A^\eps \nabla v_\eps \cdot \nabla v_\eps)$ in a weak sense.} Now, notice that
\begin{align*}
 - \int_{B_{\frac32 r}} \phi^2 (A^\eps\nabla v_\eps^2 \cdot \nabla \Gamma_\eps)\,dx
&\,\,\gio{=}\,\int_{B_{\frac32 r}} v_\eps^2 (A^\eps\nabla \Gamma_\eps \cdot \nabla \phi^2)\,dx- \int_{B_{\frac32 r}} A^\eps\nabla \Gamma_\eps \cdot \nabla (v_\eps^2 \phi^2)\,dx\\
&\leq 2\int_{B_{\frac32 r}\setminus B_r} v_\eps^2 \phi (A^\eps\nabla \Gamma_\eps \cdot \nabla \phi)\,dx,
\end{align*}
where in the second inequality we used that $-\mathrm{div}(A^\eps\nabla \Gamma_\eps)\geq 0$ in $\R^n, \gio{(\phi v_\eps)^2 \geq 0}$ and $\phi$ has compact support. Ultimately, by collecting the previous inequality and letting $\eps\to 0^+$, we finally get
\be\label{dis.estimate}
\int_{B_{\frac32 r}} \Gamma(A\nabla v \cdot \nabla v)\phi^2 \,dx \leq  \int_{B_{\frac32 r}\setminus B_r} v^2 \phi (A\nabla \Gamma\cdot \nabla \phi)\,dx - \int_{B_{\frac32 r}\setminus B_r} \phi \Gamma(A\nabla v^2\cdot \nabla \phi)\,dx,
\ee
where $v=u-u_{B_r}$. On one hand, we can estimate the last term in \eqref{dis.estimate} by applying the Young's inequality, that is
\begin{align*}
 \left|\int_{B_{\frac32 r}\setminus B_r} v\phi \Gamma (A\nabla v \cdot \nabla \phi)\,dx\right| &\leq  \frac{L C_2}{r^{n-2}}\int_{B_{\frac32 r}\setminus B_r} v^2 |\nabla \phi|^2 \,dx +
 \int_{B_{\frac32 r}\setminus B_r} \Gamma \phi^2  (A\nabla v \cdot \nabla v) \,dx\\
 &\leq  \frac{C(n,\lambda,L)}{r^{n}}\int_{B_{\frac32 r}\setminus B_r} v^2  \,dx +
 \int_{B_{\frac32 r}\setminus B_r} \Gamma (A\nabla v \cdot \nabla v)  \,dx,
\end{align*}
where in the first inequality we used the upper bound of Proposition \ref{estimate.fund.prop} and in the last one that $\phi \in [0,1]$. On the other hand, by Proposition \ref{moser.prop}, there exists an index $j \in 1+\N$, depending only on $n$, such that $4\leq k_j$ (see \eqref{moser.eq}) and so
$$
\left(\int_{B_{\frac32 r}\setminus B_r} v^{4}\,dx\right)^{\frac{1}{2}} \leq
C(n) r^{n\frac{k_j-4}{2k_j}}\left(\int_{B_{\frac32 r}\setminus B_r} v^{k_j}\,dx\right)^{\frac{2}{k_j}}
\leq C(n,\lambda,L) r^{-n/2}\int_{B_{2r}\setminus B_r}v^2 \,dx.
$$
Therefore, by Cauchy-Schwarz inequality and Lemma \ref{fundamental.lemma} it follows that
\begin{align*}
\int_{B_{\frac32 r}\setminus B_r}v^2 \phi (A\nabla \Gamma \cdot \nabla \phi)\,dx &\leq  \frac{C(n,L)}{r}\int_{B_{\frac32 r}\setminus B_r}v^2  |\nabla \Gamma|\,dx\\
& \leq \frac{C(n,L)}{r}\left(\int_{B_{\frac32 r}\setminus B_r}v^4  \,dx\right)^{1/2}\left(\int_{B_{\frac32 r}\setminus B_r}  |\nabla \Gamma|^2\,dx\right)^{1/2}\\
&\leq \frac{C(n,\lambda,L)}{r^n}\int_{B_{2r}\setminus B_r}v^2\,dx.
\end{align*}
Since $v=u-u_{B_r}$, by collecting all the previous estimates and applying the Poincar\'{e} inequality, we finally deduce
\begin{align}\label{eq}
\begin{aligned}
\int_{B_{r}}\Gamma  (A\nabla u \cdot \nabla u) \,dx &\leq  \frac{C(n,\lambda,L)}{r^{n}}\int_{B_{2 r}\setminus B_r} (u-u_{B_r})^2  \,dx +
 \int_{B_{\frac32 r}\setminus B_r} \Gamma (A\nabla u \cdot \nabla u) \,dx\\
&\leq \frac{C(n,\lambda,L)}{r^{n-2}}\int_{B_{2 r}\setminus B_r} A\nabla u\cdot \nabla u  \,dx +
 \int_{B_{2 r}\setminus B_r} \Gamma (A\nabla u \cdot \nabla u) \,dx\\
&\leq \tilde{C}(n,\lambda,L)
 \int_{B_{2 r}\setminus B_r} \Gamma (A\nabla u \cdot \nabla u) \,dx
\end{aligned}
\end{align}
where in the last inequality we used the lower bound of Proposition \ref{estimate.fund.prop}.\\
Therefore, by adding to \eqref{eq} the term $\tilde{C}\int_{B_{r}}\Gamma (A\nabla u\cdot \nabla u)\,dx$, we get
$$
\int_{B_r} \Gamma (A\nabla u \cdot \nabla u)  \,dx \leq
t
 \int_{B_{2 r}} \Gamma (A\nabla u \cdot \nabla u)  \,dx,
$$
with $t = \tilde{C}/(\tilde{C}+1) \in (0,1)$. Iterating the previous inequality, by considering integer $k$ such that $2^{-(k+1)}R<r\leq 2^{-k}R$, we deduce
\be\label{quasi}
\int_{B_r} \Gamma  (A\nabla u \cdot \nabla u)  \,dx \leq 2^\alpha\left(\frac{r}{R}\right)^{2\alpha}
\int_{B_R} \Gamma  (A\nabla u \cdot \nabla u)  \,dx \quad \mbox{for } 0<r< R\leq \frac12,
\ee
where $\alpha = -\frac{1}{2}\log_2(t)$. Finally, by applying the Poincar\'{e}'s inequality in the left hand side of \eqref{quasi} and \eqref{e:cacco} in its right hand side, it follows that
\be\label{otherwise}
\int_{B_r} (u-u_{B_r})^2 \,dx \leq C(n,  \lambda,L) r^{n+2\alpha}
\int_{B_{1}} u^2 \,dx,
\ee
for $r\in (0,1/2)$, and so the last part of the statement easily follows by Campanato's embedding.
\end{proof}
\begin{remark}\label{consid1}
In \cite{fabes-stroock} the authors derived both Nash's theorem on
continuity and Moser's Harnack principle by combining \eqref{estimate.fundamental} with an argument inspired by the works of N.V. Krylov and M.N. Safanov. In Theorem \ref{new.way.intro}, the intention is to simplify this approach in the elliptic case by constructing an $n$-dimensional version of the well-known proof of the H\"{o}lder continuity for solutions to \eqref{equation} in $\R^2$, based on the K.O. Widman’s hole-filling technique \cite{widman} and the Caccioppoli's inequality.
\end{remark}
Inspired by the analysis developed in \cite[Theorem 5.1]{CTV} and \cite[Theorem 1.3]{NTTV} for the study of singular-perturbation problem with regular coefficients, in the following result we show how the existence of an optimal Liouville exponent $k^*(n,\lambda,L)\in(0,1)$ among entire solutions to \eqref{equation} implies the local H\"{o}lder regularity for every $\alpha \in (0,k^*(n,\lambda,L))$.
\begin{definition}\label{def.optimal}
\gio{We say that $k^*_{\mathcal{F}}\in \R$ is the optimal Liouville exponent in $\mathcal{F}\subseteq M(\lambda,L)$ if it is the supremum over all exponents $\gamma>0$ with the property that if $u\in H^1_\loc(\R^n)$ is a solution of \eqref{equation.global} with $A \in \mathcal{F}$, such that
\be\label{e:gamma}
|u(x)|\leq C\left(1+|x|^{\gamma}\right)\quad\mbox{in } \R^n,
\ee
for some $C>0$, then $u$ is constant in $\R^n$.}
\end{definition}
Through this Section, we discuss two specific cases of $\mathcal{F}\subseteq M(\lambda,L)$:
\begin{itemize}
  \item if $\mathcal{F}=M(\lambda,L)$, we simply write $k^*=k^*(n,\lambda,L)$;
  \item if $\mathcal{F}=\{A\}$ for some $A\in M(\lambda,L)$, we simply say that $k^*_A \in (0,1]$ is the optimal Liouville exponent associated to $A$.
\end{itemize}
Despite the well-known implication of the Harnack's inequality in terms of the behavior of entire solutions at infinity (see for example \cite[Section 6]{moser1961}), the opposite implication of Theorem \ref{liovu.intro} was not known for divergence equation associated to $A \in M(\lambda,L)$.
\begin{proof}[Proof of Theorem \ref{liovu.intro}]
To simplify the notations, through the proof we set $k^*:=k^*(n,\lambda,L), \alpha \in (0,k^*), R:=3/4$ and
$$
D(x_0,u,r):=\frac{1}{r^{n-2+2\alpha}}\int_{B_r(x_0)}|\nabla u|^2 \,dx.
$$
Now, let $\eta_\eps \in C^\infty_c(B_\eps)$ be a standard mollifier and consider the regularized matrix $A^\eps := (a_{ij}^\eps)_{ij} \in M(\lambda,L)$ with entries $a_{ij}^\eps :=a_{ij}\ast\eta_\eps$. Then, set $u_\eps\in C^\infty(B_1)$ to be the unique solution to
$$
\mathrm{div}(A^\eps \nabla u_\eps)= 0 \quad\mbox{in }B_1,\qquad
u_\eps = u \quad\mbox{on }\partial B_1,
$$
approximating the function $u$ in the strong topology of $H^1(B_1)$, as $\eps\to 0^+$. \gio{By combining the Campanato's embedding with the strong convergence of the approximation $u_\eps$, it is sufficient to show
\be\label{e:claim}
\sup_{x_0 \in \overline{B_R}}\sup_{r \leq \frac{1-R}{2}}\left(R-|x_0|\right)^{n-2+2\alpha}D(x_0,u_\eps,r)\leq C
\ee
with $C>0$ depending only on $n,\lambda$ and $L$. Indeed, by exploiting the uniform-in-$\eps$ bound \eqref{e:claim} for $x_0 \in B_{1/2}\subset B_R, r\leq 1/8$, we get the claimed result.\\
By contradiction, suppose there exist $A_j:=A_{\eps_j}\in M(\lambda,L), u_{j}:=u_{\eps_j}\in H^1(B_1)$ a smooth solutions to \eqref{equation}, such that
$$
\lVert u_{j}\rVert_{L^\infty(B_1)}\leq 1,\qquad \mbox{and}\qquad
L^2_j := \sup_{x_0 \in \overline{B_R}}\sup_{r \leq \frac{1-R}{2}}\left(R-|x_0|\right)^{n-2+2\alpha}D(x_0,u_j,r)
\to +\infty.
$$
By Caccioppoli inequality, we know that $D(0,u_j,7/8)\leq C$ for some $C>0$ depending only on $n,\lambda$ and $L$. Moreover, there exists some maximal points $x_j \in \overline{B_R}, r_j\leq (1+R)/2$ such that
\be\label{e:mi}
(R-|x_0|)^{n-2+2\alpha}D(x_0,u_j,r)\leq
(R-|x_j|)^{n-2+2\alpha}D(x_j,u_j,r_j),
\ee
for every $x_0 \in \overline{B_R}, r\leq (1+R)/2$. On the other hand, if we set $\overline{R}_j := (R-|x_j|)/r_j$, we have
\begin{align*}
L_j^2 &= (R-|x_j|)^{n-2+2\alpha}D(x_j,u_j,r_j)\\
&\leq D(0,u_j,7/8) \left(\frac{7(R-|x_j|)}{8 r_j}\right)^{n-2+2\alpha} 
\leq C \overline{R}_j^{n-2+2\alpha}
\end{align*}
where, in light of the absurd hypotheses, it implies $\overline{R}_j \to +\infty$. We stress that
\be\label{e:ni}
\frac{R-|x_j|}{2} 
 \leq R-|x_0| \leq \frac{3(R-|x_j|)}{2},\qquad \mbox{for }x_0 \in B_{\frac{R-|x_j|}{2}}=B_{\frac12 r_j \overline{R}_j} .
\ee

\noindent Step 1 - Blow-up sequence. Let us introduce the following blow-up sequence
    \be\label{blowup}
    \overline{u}_{j}(x) := \frac{u_j(x_j + r_j x) - u_j(x_j)}{r_j^\alpha D(x_j,u_k,r_j)^{1/2}}\quad\mbox{for }x \in B_{x_j,r_j}=\frac{B-x_j}{r_j},
    \ee
    such that $\overline{u}_j(0)=0$. Moreover, by direct computation, we have $\mathrm{div}(\overline{A}_j(x) \nabla \overline{u}_j)=0$ in $B_{x_j,r_j}$, with $\overline{A}_j(x) := A_j(x_j+r_j x) \in M(\lambda,L)$, and
$$
D(z,\overline{u}_j,\rho)=\frac{D(x_j+r_j z, u_j, \rho r_j)}{D(x_j,u_j,r_j)},\qquad \mbox{for }z,\rho\mbox{ such that }B_\rho(z) \subset B_{x_j,r_j}.
$$
Thus, we deduce that $D(0,\overline{u}_j,1 ) = 1$ and, by \eqref{e:mi}-\eqref{e:ni}, that for every $z \in B_{\overline{R}_j}, \rho < (1-R)/r_j$ we have
\be\label{unif}
D(z,\overline{u}_j,\rho) = \frac{D(x_j+r_j z, u_j, \rho r_j)}{D(x_j,u_j,r_j)}\leq \left(\frac{R-|x_j|}{R-|x_j+r_j z|}\right)^{n-2+2\alpha}\leq 2^{n-2+2\alpha}.
\ee
On the other hand, by applying the Caccioppoli inequality to the rescaled equation, we get
\be\label{no-trivial}
1 = D(0,\overline{u}_j,1) 
\leq C(n,\lambda,L) \int_{B_2} \overline{u}_j^2 \,dx,
\ee
for every $j\geq j_0$ such that $B_r\subset B_{x_j,r_j}$.\\

 \noindent    Step 2 - Compactness. Fixed $\rho>0$ there exists $j_0>0$ such that $\rho< \overline{R}_j$ and $\rho<(1-R)/(r_j)$, for every $j\geq j_0$. Therefore, by rephrasing \eqref{unif}, we get
 $$
 \sup_{z \in \overline{B}_\rho}\sup_{r\leq \rho/2}D(z,u_j,r)\leq 2^{n-2+2\alpha},\qquad \mbox{for }j\geq j_0,
 $$}
 which implies, by Campanato's embedding, that the $\alpha$-H\"{o}lder seminorm of $\overline{u}_j$ is uniformly bounded in $B_\rho$, for $j$ sufficiently large. Since $\overline{u}_j(0)=0$ we infer, by a diagonal argument, that up to a subsequence $\overline{u}_j$ converges weakly in $H^1_\loc(\R^n),$ strongly in $L^2_\loc(\R^n)$ and uniformly on every compact set, to some $\overline{u}\in H^1_\loc(\R^n)\cap C^{0,\alpha}_\loc(\R^n)$. Moreover, by combining the strong convergence in $L^2_\loc(\R^n)$ with \eqref{no-trivial}, we deduce that $\overline{u}$ is a non-trivial function satisfying $\norm{\overline{u}}{L^2(B_2)}\geq C>0$.\\

We can conclude Step 2 by proving that the blow-up limit $\overline{u}$ is a weak solution of an elliptic PDE of the form \eqref{equation.global}, for some matrix in $M(\lambda,L)$. Indeed, by the compactness property \cite[Proposition 3]{spagnolo1} of $M(\lambda,L)$, the sequence $(\overline{A}_j)_j$ admits a $G$-convergent subsequence on every compact subset of $\R^n$ and Proposition \ref{closed} guarantees the existence of $\overline{A}\in M(\lambda,L)$ such that
\be\label{eq.limite}
\int_{\R^n}\overline{A} \nabla \overline{u} \cdot \nabla \varphi\, \,dx= 0 \quad \mbox{for every }\varphi \in C^\infty_c(\R^n),
\ee
as we claimed.\\

\noindent  Step 3 - Liouville theorem.
Lastly, by the global H\"{o}lder regularity of $\overline{u}$, since $\overline{u}(0)=0$ we infer that
$$
  |\overline{u}(x)|\leq C\abs{x}^\alpha\quad\mbox{in }\R^n.
$$
Therefore, since $\alpha < k^*$, by definition of optimal Liouville exponent in $M(\lambda,L)$, we deduce that $\overline{u}\equiv 0$ in $\R^n$, in contradiction with non-triviality of the blow-up limit due to \eqref{no-trivial}.
\end{proof}
\begin{remark}
  The previous result can be rephrased as follows in the case of a specific equation of the form \eqref{equation}: given $A\in M(\lambda,L)$ let $\mathcal{F}_A\in M(\lambda,L)$ be the family of all the possible $G$-limits of the matrices $(A(x_k+r_k \cdot))_k$ with $x_k \in B_{1/2},r_k \searrow 0^+$. Then, the same conclusion holds by considering the optimal Liouville exponent of the family $\mathcal{F}_A$.
\end{remark}
We conclude the Section with a simple application of the study proposed in the paper. In \cite[Secion 2]{serrin-weinberger} the authors proved the existence an optimal Liouville-type exponent $k^*(n,\lambda,L)$ in $M(\lambda,L)$ by using the iteration argument introduce by J. Moser in \cite[Section 6]{moser1961} \gio{(see also \cite{Chipot} where a similar result is proved in the context of "autonomous" variational problem)}. Since the explicit expression of the optimal Liouville exponent $k^*$ is still an open problem, by using an argument of \cite{finn-gilbarg} they declare that if a matrix $A\in M(\lambda,L)$ converges to a constant matrix at infinity\footnote{On page 262 in \cite{serrin-weinberger}, the authors write that ``If the coefficients $a_{ij}$ tend to a limit at infinity, a simple change of variables reduces the problem to the case where $\lambda/L$ is arbitrarily near one. In this case Finn and Gilbarg showed...''. By reading the results in \cite{finn-gilbarg} we think that their statement deals only with the case of matrices $A \in M(\lambda,L)$ with continuous differentiable coefficients, converging at infinity to some constant matrix in $M(\lambda,L)$.}
then the optimal Liouville exponent $k^*_A$ associate to $A$ satisfies
$$
k^*_A \geq \sqrt{n-1}-\frac{n-2}{2},
$$
for $n\leq 6$ (see Lemma \ref{lem.finn-gilbar}). By using the concept of blow-down sequence, we improve the previous estimate by showing that the optimal Liouville exponent of a fixed matrix can be obtained in terms of its blow-down limit at infinity. Moreover, if it converges to some constant matrix we get that $k^*_A=1$.
\begin{proposition}\label{liouv.Finn}
Let $A,A_\infty \in M(\lambda,L)$ and suppose that
  \be\label{G-limit}
  A(x_0+r x) \xrightarrow{\text{ $\,\,G\,\,$ }}
  A_\infty(x)\quad\mbox{as}\quad r\to +\infty
  \ee
  for some $x_0 \in \R^n$. Thus, given $k^*_A,k^*_\infty \in(0,1]$ respectively the optimal Liouville exponent associated to $A$ and $A_\infty$, we have
  $$
  k^*_A \geq k^*_\infty.
  $$
  In particular, if $A_\infty \in M(\lambda,L)$ is a constant matrix we get $k^*_A=1$.
\end{proposition}
\begin{proof}
Assume by contradiction that
  $$
  k^*_A < k^*_\infty
  $$
  and, by the definition of $k^*_A$, that there exists a nontrivial solution $u \in H^1_\loc(\R^n)$ to \eqref{equation.global}, such that $u(0)=0$ \gio{and \eqref{e:gamma} holds true with $\gamma \leq k^*_A$}. It is not restrictive to assume that $x_0=0$: indeed, by a simple translation we get that $v(x):=u(x+x_0)$ is solution to
$$
\mathrm{div}(\tilde{A}\nabla v)=0\quad\mbox{in }\R^n, \quad\mbox{with }\tilde{A}(x):=A(x+x_0)\in M(\lambda,L).
$$
On the other hand, by the existence of the $G$-limit \eqref{G-limit} we get that
$$
\tilde{A}(rx) = A(x_0+rx) \xrightarrow{\text{ $\,\,G\,\,$ }}
  A_\infty(x)\quad\mbox{as}\quad r\to +\infty.
$$
Now, we divide the remaining part of the proof in two steps.\\

    {\it Step 1 - Blow-down sequence in the general case.}
 Given $r\geq 1$, set $A_r(x):=A(rx)$ and
  $$
  \theta(r):=\sup_{\rho\geq r}\frac{1}{\rho^{\gamma}}\norm{u}{L^\infty(B_\rho)}.
  $$
  Notice that the function $\theta$ is non-increasing and $\theta(r) \in (0,+\infty)$. Moreover, for every integer $j \in \N$, there exists $r_j\geq j$ such that
  $$
  \frac{1}{r_j^{\gamma}}\norm{u}{L^\infty(B_{r_j})} \geq \frac12\theta(j)\geq \frac12\theta(r_j).
  $$
  Thus, consider the blow-down sequence
  $$
  u_j(x):=\frac{1}{r_j^\gamma \theta(r_j)}u(r_j x)
  $$
  of solutions to
  $$\mathrm{div}(A_{r_j}\nabla u_j)=0\qquad\mbox{in }\R^n,
  $$ satisfying the growth estimate
  \be\label{xj1}
  \norm{u_j}{L^\infty(B_\rho)}= \frac{1}{r_j^\gamma \theta(r_j)}
  \norm{u}{L^\infty(B_{\rho r_j})}\leq \rho^\gamma,\quad\mbox{for }\rho\geq 1
  \ee
  and the non-degeneracy condition
  \be\label{xj2}
  \norm{u_j}{L^\infty(B_1)} =   \frac{\norm{u}{L^\infty(B_{r_j})}}{r_j^{\gamma}\theta(r_j)}\geq \frac12.
  \ee
 On the other hand, by the Caccioppoli inequality
    $$
\int_{B_{\rho}}{\abs{\nabla u_j}^2 \,dx}\leq \frac{C(n,\lambda,L)}{\rho^2}  \int_{B_{2\rho}}{u_j^2 \,dx} \leq C(n,\lambda,L) \rho^{n-2+2\gamma}, \quad\mbox{for }\rho\geq 1
$$
we get a uniform bound in $H^1(B_\rho)$ for the sequence $(u_j)_j$. Therefore, by the De Giorgi-Nash-Moser theorem (see otherwise Theorem \ref{new.way.intro}), the sequence $(u_k)_k$ converges, up to a subsequence, to some $u \in H^1_\loc(\R^n)$ uniformly on every compact set of $\R^n$ and strongly in $L^2_\loc(\R^n)$.\\
Finally, by combining Theorem \ref{spagno1} with \eqref{G-limit}, $u_\infty$ is a weak-solution to $$
\mathrm{div}(A_\infty \nabla u_\infty)=0\qquad\mbox{in }\R^n,
$$
satisfying $u_\infty(0)=0$ and
\be\label{nro}
\norm{u_\infty}{L^\infty(B_1)}\geq \frac12 \quad\mbox{and}\quad
\norm{u_\infty}{L^\infty(B_\rho)}\leq \rho^\gamma\quad\mbox{for every }\rho\geq 1.
\ee
Notice that, by \cite[Theorem 4]{kozlov} the existence of a $G$-limit \eqref{G-limit} as $r\to +\infty$ implies that the coefficients of $A_\infty$ are zero-homogeneous. Finally, the result follows by the definition of optimal Liouville exponent $k^*_\infty$, indeed since $\gamma\leq k^*_A<k^*_\infty$ we get that $u_\infty\equiv 0$, in contradiction with \eqref{nro}.\\

Whereas, if $A_\infty$ has constant coefficients, we can repeat the same argument by considering the ``freezed'' equation
$$
  \mathrm{div}(\tilde A(x)\nabla \tilde u) = 0 \quad\text{in $\R^n$},
$$
where $\tilde A(x):= (A^{-1/2}_\infty)^T A(A^{-1/2}_\infty x) A^{-1/2}_\infty$ and $\tilde u(x):=u(A^{1/2}_\infty x)$. Indeed, in such case we get that
$$
\tilde A(rx) \xrightarrow{\text{ $\,\,G\,\,$ }}
  \mathrm{Id} \quad\mbox{as}\quad r\to +\infty,
$$
from which it follows that $k^*_A\geq 1$.
\end{proof}

\noindent \textbf{Acknowledgments.} We are grateful to S. Terracini and B. Velichkov for useful discussions on the problem and to the referees for their observations and suggestions regarding the main Theorems.

\smallskip

{\noindent \textbf{Statements and Declarations.} The authors have no relevant financial or non-financial interests to disclose.}

\smallskip

{\noindent \textbf{Data availability.} Data sharing not applicable to this article as no datasets were generated or analysed during the current study.}

\end{document}